\documentclass[oneside,reqno,12pt]{amsart}

\usepackage{amsmath,amssymb,amsfonts}
\usepackage{kpfonts}
\usepackage{mathtools}
\usepackage{bbm}
\usepackage{hyperref}
\usepackage[foot]{amsaddr}

\usepackage{bbm}

\newtheorem{thm}{Theorem}
\newtheorem{lemma}{Lemma}

\theoremstyle{definition}
\newtheorem{defin}{Definition}

\theoremstyle{remark}

\theoremstyle{remark}
\newtheorem{remark}{Remark}

\newtheorem{assx}{Assumption A\hspace*{-4pt}}

\newcommand{\ii}{\ensuremath{\mathbf{1}}}
\newcommand{\rr}{\ensuremath{\mathbb{R}}}
\newcommand{\zz}{\ensuremath{\mathbb{Z}}}

\newcommand{\xx}{{\sf{X}}}
\newcommand{\pp}{{\sf{P}}}

\newcommand{\llf}{{\sf{L}}}
\newcommand{\bb}{\mathcal{B}}
\newcommand{\ff}{\mathcal{F}}

\newcommand{\aaa}{{\sf A}}
\newcommand{\pt}{\partial}

\begin{document}

\title[Asymptotic properties of the heat equation]{Asymptotic properties of the heat equation driven by stochastic measure}

\author{Vadym Radchenko}

\address{Department of Mathematical Analysis, Taras Shevchenko National University of Kyiv, 01601 Kyiv, Ukraine. Email: vadymradchenko@knu.ua}

\begin{abstract}
For the heat equations driven by a stochastic measure on $[-L, L]$ with the Dirichlet boundary condition, we prove that solutions tend to the solution of the heat equations defined on $\rr$ as $L\to \infty$. The estimate of the convergence rate is obtained. For a stochastic measure, we assume the $\sigma$-additivity in probability only.
\end{abstract}


\maketitle

\section{Introduction}\label{scintr}
In this paper, we consider the stochastic heat equation with the Dirichlet boundary condition, which can formally be written as
\begin{equation}\label{eq:burggl}
\begin{split}
\frac{\pt u^L}{\pt t}=\frac{\pt^2 u^L}{\pt x^2}+f(t,x,u^L(t,x))+\sigma(t, x)\frac{\pt  \mu}{\pt x},\\
u^L(t,-L)=u^L(t,L)=0,\quad u^L(0,x)=u_0(x),
\end{split}
\end{equation}
where $(t, x)\in [0, T]\times[-L,L]$, and $\mu$ is a stochastic measure defined on $\bb(\rr)$ (Borel
$\sigma$-algebra in $\rr$). For $\mu$ we assume only $\sigma-$additivity in probability, assumptions for $f,\ \sigma$
and $u_0$ are given in Section~\ref{sc:prob}. We study the solution to the formal equation~\eqref{eq:burggl} in the mild
form (see~\eqref{eq:bemfgl} below).

We study the convergence as $L\to \infty$, and prove that $u^L$ tends to the solution to the equation
\begin{equation}\label{eq:burgg}
\frac{\pt u}{\pt t}=\frac{\pt^2 u}{\pt x^2}+f(t,x,u(t,x))+ \sigma(t, x)\frac{\pt  \mu}{\pt x},\quad u(0,x)=u_0(x),
\end{equation}
where $(t, x)\in [0, T]\times\rr$. The solution to equation~\eqref{eq:burgg} is also considered in the mild form (see~\eqref{eq:bemfg} below).

For an equation driven by space-time white noise, this problem was studied in~\cite{candil26} where the Neumann and mixed boundary conditions were also considered.

Theory of equations driven by stochastic measures is developed in~\cite{radbook}. Existence and uniqueness of mild solutions to~\eqref{eq:burggl} and ~\eqref{eq:burgg}, their boundedness and H\"{o}lder regularity, are proved in~\cite{radt26} and~\cite{radt24} respectively. Also, generalization of equation~\eqref{eq:burgg} is considered in~\cite{radm23}, where the $\llf_2$-continuity of the solution was proved.

Equations driven by the space-time white noise are considered in detail in~\cite{dalbook}.

The rest of the paper is organized as follows. Section~\ref{sc:prel} contains the basic information concerning the integrals with respect to stochastic measures. In Section~\ref{sc:prob}, we formulate the problem and our main assumptions. In Section~\ref{sc:auxl} we obtain an estimate for the difference of the stochastic integrals from our equations. Section~\ref{sc:mainr} contains the main result of the paper.

\section{Preliminaries}\label{sc:prel}

Let $\llf_0=\llf_0(\Omega, {\ff}, {\pp})$ be the set of all real-valued
random variables defined on the complete probability space $(\Omega, {\ff}, {\pp} )$. Convergence in $\llf_0$ means the convergence in probability. Let ${\xx}$ be an arbitrary set and ${\bb}$ a $\sigma$-algebra of subsets of ${\xx}$.

\begin{defin}\label{def:stme}
A $\sigma$-additive mapping $\mu:\ {\bb}\to \llf_0$ is called {\em stochastic measure} (SM).
\end{defin}

We do not assume the moment existence or martingale properties for SM. We can say that $\mu$ is an $\llf_0$--valued measure.

Important examples of SMs are orthogonal stochastic measures, $\alpha$-stable random measures defined on a $\sigma$-algebra for $\alpha\in (0,1)\cup(1,2]$ (see \cite[Chapter 3]{samtaq}). Many examples of the SMs on $\bb(\rr)$ may be given by the Wiener-type integral
\begin{equation}\label{eq:muax}
\mu(A)=\int_{\rr} {\ii}_A(x)f(x)\,{d}Y_x,
\end{equation}
where $f$ is a deterministic function, integrable on $\rr$. For instance, we can take in~\eqref{eq:muax} any square integrable martingale $Y_x$ or $Y_x=W_x^H$~-- the fractional Brownian motion with Hurst index $H>1/2$.

Other examples of SMs can be found in Section 1.2.1~\cite{radbook}.

For deterministic measurable functions $f:\xx\to\rr$, an integral of the form $\int_{\xx}f\,d\mu$ is defined and studied
in~\cite[Chapter 7]{kwawoy}, ~\cite[Chapter 1]{radbook}. In particular, every bounded measurable $f$ is integrable with respect to any~$\mu$.

Below, we will use the following statement.

\begin{lemma} \label{lmfkmu} (Lemma 3.1~\cite{radt09}, Lemma 2.3~\cite{radbook}) Let $\mu$ be an SM on $(\xx,\bb)$, $f_k:\ {\xx}\to \rr,\ k\ge 1$, be measurable functions such that
$\hat{f}(x)=\sum_{k=1}^{\infty} |{f_k}(x)|$ is integrable with respect to~$\mu$. Then
\begin{equation*}
\sum_{k=1}^{\infty}\Bigl(\int_{\xx} f_k\,d\mu \Bigr)^2<\infty\quad
\mbox{\textrm ~a.~s.}
\end{equation*}
\end{lemma}

For any $j\in \zz$ and all $n\ge 0$, put
\[
d_{kn}^{(j)}=j+k 2^{-n},\quad 0\le k\le 2^n,\quad \Delta_{kn}^{(j)}=(d_{(k-1)n}^{(j)}, d_{kn}^{(j)}],\quad
1\le k\le 2^n\, .
\]

The following lemma is a key tool for estimating a version of the stochastic integral. In the sequel, we assume that SM $\mu$ is defined on Borel subsets of $\rr$.

\begin{lemma}\label{lmessih} (Lemma 2.8~\cite{radbook})
Let $Z$ be an arbitrary set, and the function $q(z,x):
Z\times [j, j+1]\to\rr$ be such that all paths $q(z,\cdot)$ are continuous on $[j,j+1]$, $\aaa\subset(j,j+1]$ be a Borel set. Denote
\[
q_n(z,x)=\sum_{1\le k\le 2^n}q(z,{d}_{(k-1)n}^{(j)}){\bf 1}_{\Delta_{kn}^{(j)}}(x).
\]
Then the random function
\[
\eta(z)=\int_{\aaa}q(z, x)\,d\mu(x),\ z\in Z,
\]
has a version
\begin{equation}\begin{split}
\widetilde{\eta}(z)  = \int_{\aaa}q_0(z,x)\,d\mu(x)
+\sum_{n\ge 1}\Bigl(\int_{\aaa}q_n(z,x)\,d\mu(x)-\int_{\aaa}q_{n-1}(z,x)\,d\mu(x)\Bigr)\label{eqvers}
\end{split}
\end{equation}
such that for all $\beta>0$, $\omega\in\Omega$, $z\in Z$
\begin{equation}\label{eqesqm} \begin{split}
|\widetilde{\eta}(z)| \le  |q(z,j)\mu(\aaa)|\\
  + \Bigl\{\sum_{n\ge 1}2^{n\beta}
\sum_{1\le k\le 2^{n}}|q(z,d_{kn}^{(j)})-q(z,d_{(k-1)n}^{(j)})|^2\Bigr\}^{1/2}\\
\times\Bigl\{\sum_{n\ge 1}2^{-n\beta}\sum_{1\le k\le 2^{n}}|\mu(\Delta_{kn}^{(j)}\cap\aaa)|^2\Bigr\}^{1/2}.
\end{split}
\end{equation}
\end{lemma}

In the sequel, by $C$ and $C(\omega)$ we denote the constants and random constants, respectively, whose exact values are not important. Notations of the kind $C_\lambda$, $C_{\delta}$ means that these constants may depend of $\lambda$, $\delta$ respectively.

\section{The problem}
\label{sc:prob}

We consider~\eqref{eq:burggl} in the mild form
\begin{equation}\label{eq:bemfgl}
\begin{split}
u^L(t,x)=\int_{[-L,L]} G^L(t,x,y)u_0(y)\,dy\\
+\int_0^t \int_{[-L,L]}G^L(t-s,x,y) f(s,y,u^L(s,y))\,dy\,ds\\
+\int_{[-L,L]} \int_0^t G^L(t-s,x,y) \sigma(s,y)\,ds\,d\mu(y),
\end{split}
\end{equation}
where equality holds a.~s. for each $(t,x)\in [0,T]\times [-L,L]$. Here ${G}^L(t, x, y)$ is the fundamental solution of the heat equation with the Dirichlet boundary condition
\begin{equation*}
\dfrac{\pt u}{\pt t}=\dfrac{\pt^2 u}{\pt x^2},\quad (t,x)\in [0,T]\times [-L,L],\quad u(t,-L)=u(t,L)=0.
\end{equation*}
It is known that
\begin{equation*}
G^L(t, x, y)=\sum_{n=-\infty}^{\infty}(p(t,x-y+4nL)-p(t,x+y+(4n+2)L)),
\end{equation*}
where
\begin{equation*}
p(t,x)=\dfrac{1}{\sqrt{4\pi t}}\exp\Bigl\{-\dfrac{x^2}{4t}\Bigr\}
\end{equation*}
is a standard heat kernel (see, for example, \cite[Section 8.4.2]{pinsky}, \cite[Lemma 1.4.1]{dalbook}).

We will refer to the following assumptions on elements of~\eqref{eq:bemfgl}.

\begin{assx}\label{assxu} $u_0(y)=u_0(y, \omega):{\rr}\times\Omega\to\rr$ is measurable and $|u_0(y)|\le C_{u_0}(\omega)$
for some random constant $C_{u_0}(\omega)$.
\end{assx}

\begin{assx}\label{assxfgb}  $f(s, y, v):[0,T]\times{\rr}\times\rr\to\rr$ is measurable, and bounded:
\[
|f(s, y, v)|\le C_{f}
\]
for some constant $C_{f}$.
\end{assx}

\begin{assx}\label{assxfgyv}  $f(s, y, v)$ is globally Lipschitz in $v,y$:
\[
|f(s, y_1, v_1)-f(s, y_2, v_2)|\le L_{f} (|v_1-v_2|+ |y_1-y_2|),\quad y_i,\ v_i \in\rr,
\]
for some constant $L_{f}$.
\end{assx}

\begin{assx}\label{assxs}
$\sigma(s,y):[0, T]\times{\rr}\to\rr$ is measurable, and
\[
|\sigma(s,y)|\le C_{\sigma},\quad |\sigma(s,y_1)-\sigma(s,y_2)|\le L_{\sigma}|y_1-y_2|^{\beta(\sigma)}
\]
for some $1/2<\beta(\sigma)<1$ and constants $C_{\sigma},\ L_{\sigma}$.
\end{assx}

\begin{assx}\label{assxm} $|y|^{\tau}$ is integrable w.~r.~t. $\mu$ on~$\rr$ for some $\tau>1/2$.
\end{assx}

Assumptions A\ref{assxu} -- A\ref{assxs} and Theorem~5.1~1)~\cite{radt26} imply that \eqref{eq:bemfgl} has a unique bounded solution.

Also, we consider~\eqref{eq:burgg} in the mild form
\begin{equation}\label{eq:bemfg}
\begin{split}
u(t,x)=\int_{\rr} p(t,x-y)u_0(y)\,dy
+\int_0^t \int_{\rr}p(t-s,x-y) f(s,y,u(s,y))\,dy\,ds\\
+\int_{\rr} \int_0^t p(t-s,x-y) \sigma(s,y)\,ds\,d\mu(y).
\end{split}
\end{equation}

The existence and uniqueness of the bounded solution to~\eqref{eq:bemfg} follows from Theorem~5.1~1)~\cite{radt24}.

This paper aims to prove that
\[
u^L(t,x)\to u(t,x),\quad L\to\infty,
\]
for each $(t,x)\in [0,T]\times \rr$. Mainly, we will consider $L=i\delta$ for some fixed $\delta>0$, $i\ge 1$, and obtain the estimate of $|u^L(t,x)-u(t,x)|$ for some versions of the solutions.

In the estimates below, we will use that
\begin{equation}\label{eq:estpx}
\Bigl|\frac{\pt p}{\pt y} (t-s,x-y)\Bigr|\le  \frac{C_\lambda}{t-s}
 e^{-\frac{\lambda(x-y)^2}{t-s}}
\end{equation}
for any $\lambda\in (0,1/4)$ and some constant $C_\lambda$.

\section{Auxiliary lemma}\label{sc:auxl}

In the sequel, we will use the notations $D_L=[-L,L]$, $D_L^c=\rr\setminus [-L,L]$, $\zz^*=\zz\setminus\{0\}$.

Consider the stochastic terms
\begin{equation*}
\begin{split}
\vartheta^L(t,x)=\int_{D_L} \int_0^t G^L(t-s,x,y) \sigma(s,y)\,ds\,d\mu(y),\\
\vartheta(t,x)=\int_{\rr} \int_0^t p(t-s,x-y) \sigma(s,y)\,ds\,d\mu(y).
\end{split}
\end{equation*}

In the following statement, we fix any $\delta>0$, and consider $L=i\delta$, $i\in\zz_+$.

\begin{lemma}\label{lm:zetaes}
Let Assumptions A\ref{assxs} and A\ref{assxm} hold. Then, for each $\delta>0$ there exist versions of $\vartheta^L, \vartheta$ such that for all $\omega\in\Omega$ for each $x\in D_L$, $L=i\delta$, $i\in\zz_+$, $0<t\le T$, for any $0<\lambda<\frac{1}{4}$ holds
\begin{equation}\label{eq:zetaes}
|\vartheta^L(t,x)-\vartheta(t,x)|\le  C_{\delta}(\omega) C_\lambda  e^{-\frac{\lambda (L-|x|)^2}{t}} .
\end{equation}
\end{lemma}

\begin{proof}
For $x\in D_L$, we will estimate
\begin{equation*}
\begin{split}
\vartheta^L(t,x)-\vartheta(t,x)=\int_{D_L} \int_0^t G^L(t-s,x,y) \sigma(s,y)\,ds\,d\mu(y)\\
-\int_{\rr} \int_0^t p(t-s,x-y) \sigma(s,y)\,ds\,d\mu(y)\\
=\int_{D_L}  \int_0^t (G^L(t-s,x,y) - p(t-s,x-y) )\sigma(s,y)\,ds\,d\mu(y)\\
-\int_{D_L^c} \int_0^t p(t-s,x-y) \sigma(s,y)\,ds\,d\mu(y)
:=I_1+I_2.
\end{split}
\end{equation*}
First, consider $I_1$. For $x,y\in D_L$,
\begin{equation}\label{eq:estgp}
\begin{split}
G^L(t-s,x,y) - p(t-s,x-y)=\sum_{n\in\zz^*}p(t-s,x-y+4nL)\\
-\sum_{n\in\zz^*}p(t-s,x+y+(4n+2)L)-p(t-s,x+y+2L).
\end{split}
\end{equation}

For $x, y\in D_L$ we have that $|x+y+2L|\ge L-|x|$.

We will take $i$ large enough such that $|x-y+4nL|, |x+y+(4n+2)L|>T\ge t-s,\ n\in\zz^*$ here. (By Lemma~4.2~\cite{radt24}, $\vartheta$ is bounded. By (4.9) of \cite{radt26}, each $\vartheta^{i\delta}$ is bounded. Thus, excluding the finite number of $i$ does not influence~\eqref{eq:zetaes}.)

For fixed $t, x$ set
\begin{equation}\label{eq:defq}
q(z,y)= \int_0^t p(t-s,x-y)\sigma (s,y)\,ds,\quad
z=(t,x),\ y\in D_L.
\end{equation}
Our considerations will imply that this function is continuous in~$y$; thus, the conditions of Lemma~\ref{lmessih} hold (for this purpose, we can use estimate~\eqref{est11} below).

Consider
\begin{eqnarray*}
    q(z,y+h)-q(z,y)
  = \int_0^t p (t-s,x-y)
(\sigma (s,y+h)-\sigma (s,y))\,ds\\
 + \int_0^t(p(t-s,x-y-h)-p(t-s,x-y)) \sigma (s,y+h)\,ds
:=I_{11}+I_{12}.
\end{eqnarray*}

We will use the following simple estimate for $b/t\ge 1$.
\begin{eqnarray}
\int_0^t\frac{1}{r} e^{-\frac{b}{r}} dr\stackrel{{b}/{r} =z}{=}\int_{b/t}^{\infty}\frac{1}{z} e^{-z} dz\le \int_{b/t}^{\infty}e^{-z} dz= e^{-b/t}.
  \label{gaestbt}
\end{eqnarray}
We get for $|x-y|, |x-y-h|\ge kL-|x|>T$, $h<L$
\begin{equation*}
\begin{split}
|I_{12}| \stackrel{\rm A\ref{assxs}}{\le}  C_{\sigma}\int_0^t |p(t-s,x-h-y)-p(t-s,x-y)|ds\\
\le\int_0^t \Bigl|\int_{x-h}^{x}\dfrac{\partial p(t-s,v-y)}{\partial v} \,dv  \Bigr|ds
 \stackrel{\eqref{eq:estpx}}{\leq}   \int_0^t\Bigl(\frac{C_\lambda }{t-s}
\int_{x-h}^{x}e^{-\frac{\lambda |v-y|^2}{t-s}}dv\Bigr)ds\\
 \stackrel{t-s=r}{=} C_\lambda \int_{x-h}^{x}dv \int_0^t\frac{1}{r} e^{-\frac{\lambda |v-y|^2}{r}}dr
 \stackrel{\eqref{gaestbt}}{\leq}  C_\lambda \int_{x-h}^{x} e^{-\frac{\lambda |v-y|^2}{t}}\,dv
 \le C_\lambda |h|  e^{-\frac{\lambda (kL-|x|)^2}{t}}.
\end{split}
\end{equation*}

Also,
\begin{equation}\label{est41}
\begin{split}
|I_{11}| \le  \int_0^t p(t-s,x-y)|\sigma (s,y+h)-\sigma (s,y)|\,ds\\
\stackrel{\rm A\ref{assxs}}{\le}  Ch^{\beta(\sigma)} \int_0^t\dfrac{1}{\sqrt{t-s}}e^{-\frac{|x-y|^2}{t-s}}ds\le Ch^{\beta(\sigma)} \int_0^t\dfrac{1}{\sqrt{t-s}}e^{-\frac{(kL-|x|)^2}{t-s}}ds\\
\le Ch^{\beta(\sigma)} e^{-\frac{(kL-|x|)^2}{t}}\int_0^t \dfrac{1}{\sqrt{t-s}} ds
\le Ch^{\beta(\sigma)} e^{-\frac{(kL-|x|)^2}{t}} .
\end{split}
\end{equation}

Thus, for any $0<\lambda<\frac{1}{4}$
\begin{equation}\label{est11}
 |q(z,y+h)-q(z,y)|\le  C_\lambda h^{\beta(\sigma)} e^{-\frac{\lambda(kL-|x|)^2}{t}}.
\end{equation}
Similarly, consider $p(t-s,x+y+(4n+2)L)$, $n\in\zz^*$, $p(t-s,x+y+2L)$, and for
\begin{eqnarray*}
\widetilde{q}(z,y)= \int_0^t \Bigl(\sum_{n\in\zz^*} p(t-s,x-y+4nL)-\sum_{n\in\zz^*}p(t-s,x+y+(4n+2)L)\\
-p(t-s,x+y+2L)\Bigr)\sigma (s,y)\,ds
\end{eqnarray*}
we get
\begin{equation}\label{eq:wqzh}
 |\widetilde{q}(z,y+h)-\widetilde{q}(z,y)|\le  CC_\lambda h^{\beta(\sigma)} e^{-\frac{\lambda(L-|x|)^2}{t}},\quad
z=(t,x),\ y,\ y+h\in D_L.
\end{equation}
As in~\eqref{est41}, we obtain
\begin{eqnarray*}
|{q}(z,y)|\le C C_\lambda e^{-\frac{\lambda(kL-|x|)^2}{t}}.
\end{eqnarray*}
Therefore,
\begin{eqnarray}\label{eq:wqz}
 |\widetilde{q}(z,y)|\le C C_\lambda e^{-\frac{\lambda(L-|x|)^2}{t}},\quad
z=(t,x),\ y\in D_L.
\end{eqnarray}

We take $\tau>1/2$ from A\ref{assxm} and some $\beta>0$ such that $\beta+1-2 \beta(\sigma)<0$.
From~\eqref{eqesqm}, for the version~\eqref{eqvers} of the stochastic integral (taking intervals $[j\delta,(j+1)\delta]$ instead of $[j,j+1]$ and $d_{kn}^{(j)}=(j+k 2^{-n})\delta$), we get
\begin{equation}\label{eq:estint}
\begin{split}
\Bigl|\int_{D_L}  \int_0^t (G^L(t-s,x,y) - p(t-s,x-y) )\sigma(s,y)\,ds\,d\mu(y)\Bigr|\\
=\Bigl|\int_{D_L}\widetilde{q}(z, y)\,d\eta (y)\Bigr|
\le \sum_{j\in\zz}\Bigl|\int_{(j\delta, (j+1)\delta]\cap D_L} \widetilde{q}(z, y)\,d\mu (y)\Bigr|\\
  \stackrel{\eqref{eqesqm},L=i\delta}{\le} \sum_{j: (j\delta, (j+1)\delta]\subset D_L}
    |\widetilde{q}(z,j\delta)\mu((j\delta, (j+1)\delta])|\\
  + \Bigl\{\sum_{n\ge 1}2^{n\beta}
\sum_{1\le k\le 2^{n}}|\widetilde{q}(z,d_{kn}^{(j)})-\widetilde{q}(z,d_{(k-1)n}^{(j)})|^2\Bigr\}^{1/2}\Bigl\{\sum_{n\ge 1}2^{-n\beta}\sum_{1\le k\le 2^{n}}|\mu(\Delta_{kn}^{(j)})|^2\Bigr\}^{1/2}  \\
\stackrel{\eqref{eq:wqz},\eqref{eq:wqzh}}{\le} C C_\lambda e^{-\frac{\lambda(L-|x|)^2}{t}} \Bigl[\sum_{j: (j\delta, (j+1)\delta]\subset D_L} |\mu (j\delta, (j+1)\delta])|\\
 + \sum_{j: (j\delta, (j+1)\delta]\subset D_L} \Bigl\{\sum_{n\ge 1}2^{n(\beta+1-2 \beta(\sigma))}\sum_{1\le k\le 2^{n}}
|\mu (\Delta_{kn}^{(j)})|^2 \Bigr\}^{1/2}\Bigr]\\
{\le}  C C_\lambda e^{-\frac{\lambda(L-|x|)^2}{t}}
\Bigl[ \Bigl(\sum_{j\in\zz}(|j|+1)^{2\tau}(\mu ((j\delta, (j+1)\delta]))^2
\Bigr)^{1/2}\Bigl(\sum_{j\in\zz}(|j|+1)^{-2\tau}\Bigr)^{1/2}\\
 + \Bigl(\sum_{j\in\zz}(|j|+1)^{2\tau}  \sum_{n\ge 1} 2^{n(\beta+1-2 \beta(\sigma))} \sum_{1\le k\le 2^n}
|\mu ({\Delta_{kn}^{(j)}})|^2\Bigr)^{1/2}
\Bigl(\sum_{j\in\zz}(|j|+1)^{-2\tau} \Bigr)^{1/2}\Bigr].
\end{split}
\end{equation}
Here, the sums with $\mu $ have a form
 $\sum_{l=1}^{\infty}\Bigl(\int_{\rr} f_l\,d\mu  \Bigr)^2$, where
\begin{eqnarray*}
\{f_l (y),\ l\ge 1\}  =  \{(|j|+1)^{\tau}\, \ii_{(j\delta, (j+1)\delta]}(y),\ j\in \zz\},\\
\{f_l (y),\ l\ge 1\}  =  \{(|j|+1)^{\tau} 2^{n(\beta+1-2 \beta(\sigma))/2}
\ii_{\Delta_{kn}^{(j)}}(y),\ j\in \zz,\ n\ge 1,\ 1\le k\le 2^{n}\}.
\end{eqnarray*}
We have $\sum_{l=1}^{\infty}|f_l(y)|\le C(|y|+1)^{\tau}$ in both cases. From Lemma~\ref{lmfkmu} it follows that
 \[
 \sum_{l=1}^{\infty}\Bigl(\int_{\xx} f_l\,d\mu  \Bigr)^2<+\infty\ \textrm{a.\ s.}
 \]
We take a set $\Omega_1\subset\Omega$ such that for $\omega\in\Omega_1$, this series is finite for both sequences $\{f_l\}$. Thus $\pp(\Omega_1)=1$ and
\begin{equation}\label{eq:estio}
|I_1|\le  C_{\delta}(\omega) C_\lambda  e^{-\frac{\lambda (L-|x|)^2}{t}}
\end{equation}
holds for all $L=i\delta$, $i\in \zz_+$, and $\omega\in\Omega_1$. The set $\Omega_1$ depends only on $\mu$, $\delta$ and $\beta$.

In $I_2$ we have $|x-y|\ge L-|x|$. Again we consider $q(z,y)$ from~\eqref{eq:defq}, now with $|y|>L$. Repeating the consideration from~\eqref{eq:defq} to~\eqref{est11}, we get here
\begin{equation*}
\begin{split}
|q(z,y+h)-q(z,y)|\le  C_\lambda h^{\beta(\sigma)} e^{-\frac{\lambda(L-|x|)^2}{t}},\\
 |{q}(z,y)|\le C C_\lambda e^{-\frac{\lambda(kL-|x|)^2}{t}},\quad |y|>L.
\end{split}
\end{equation*}
Similarly to \eqref{eq:estint}, we get
\begin{equation}\label{eq:estintc}
\begin{split}
\Bigl|\int_{D_L^c} \int_0^t p(t-s,x-y) \sigma(s,y)\,ds\,d\mu(y)\Bigr|\\
=\Bigl|\int_{D_L^c}{q}(z, y)\,d\eta (y)\Bigr|
\le \sum_{j\in\zz}\Bigl|\int_{(j\delta, (j+1)\delta]\cap D_L^c} {q}(z, y)\,d\mu (y)\Bigr|\\
  \stackrel{\eqref{eqesqm},L=i\delta}{\le} \sum_{j: (j\delta, (j+1)\delta]\subset D_L^c}
    |{q}(z,j\delta)\mu((j\delta, (j+1)\delta])|\\
  + \Bigl\{\sum_{n\ge 1}2^{n\beta}
\sum_{1\le k\le 2^{n}}|{q}(z,d_{kn}^{(j)})-{q}(z,d_{(k-1)n}^{(j)})|^2\Bigr\}^{1/2}\Bigl\{\sum_{n\ge 1}2^{-n\beta}\sum_{1\le k\le 2^{n}}|\mu(\Delta_{kn}^{(j)})|^2\Bigr\}^{1/2}  \\
\stackrel{\eqref{eq:wqz},\eqref{eq:wqzh},L=i\delta}{\le} C C_\lambda e^{-\frac{\lambda(L-|x|)^2}{t}} \Bigl[\sum_{j: (j\delta, (j+1)\delta]\subset D_L^c} |\mu (j\delta, (j+1)\delta])|\\
 + \sum_{j: (j\delta, (j+1)\delta]\subset D_L^c} \Bigl\{\sum_{n\ge 1}2^{n(\beta+1-2 \beta(\sigma))}\sum_{1\le k\le 2^{n}}
|\mu (\Delta_{kn}^{(j)})|^2 \Bigr\}^{1/2}\Bigr]\\
{\le}  C C_\lambda e^{-\frac{\lambda(L-|x|)^2}{t}}
\Bigl[ \Bigl(\sum_{j\in\zz}(|j|+1)^{2\tau}(\mu ((j\delta, (j+1)\delta]))^2
\Bigr)^{1/2}\Bigl(\sum_{j\in\zz}(|j|+1)^{-2\tau}\Bigr)^{1/2}\\
 + \Bigl(\sum_{j\in\zz}(|j|+1)^{2\tau}  \sum_{n\ge 1} 2^{n(\beta+1-2 \beta(\sigma))} \sum_{1\le k\le 2^n}
|\mu ({\Delta_{kn}^{(j)}})|^2\Bigr)^{1/2}
\Bigl(\sum_{j\in\zz}(|j|+1)^{-2\tau} \Bigr)^{1/2}\Bigr].
\end{split}
\end{equation}
As in \eqref{eq:estio}, we get that
\begin{equation}\label{eq:estit}
|I_2|\le  C_{\delta}(\omega) C_\lambda  e^{-\frac{\lambda (L-|x|)^2}{t}},
\end{equation}
\eqref{eq:estio} and \eqref{eq:estit} imply that \eqref{eq:zetaes} holds.
\end{proof}

\begin{remark}\label{rem:nodelta}
Let SM $\mu$ be such that $\mu((a,b])\le C(\omega)$, $a,\ b\in\rr$. Then we can obtain \eqref{eq:zetaes} for arbitrary $L\in\rr_+$, and get
\[
|\vartheta^L(t,x)-\vartheta(t,x)|\le  C(\omega) C_\lambda  e^{-\frac{\lambda (L-|x|)^2}{t}} ,
\]
where $C(\omega)$ is independent of $L$. To prove this, we can fix $\delta=1$, and note that in this case in~\eqref{eq:estint} we get additional summands
\begin{eqnarray*}
|\mu (j_0, L]\cap D_L)|+|\mu [-L, -j_0]])|\\
+\sum_{n\ge 1}2^{-n\beta}|\mu((d_{k_n n}^{(j_0)},L])|+\sum_{n\ge 1}2^{-n\beta}|\mu([-L,d_{k_n n}^{(-j_0)}])|
\end{eqnarray*}
for some $j_0$, $k_n$ that depend on $L$ (and similar summands in \eqref{eq:estintc}). If $\mu((a,b])\le C(\omega)$, then these sums are uniformly bounded.
\end{remark}

\section{Main result}\label{sc:mainr}

\begin{thm}
Let Assumptions A\ref{assxu} -- A\ref{assxm} hold, $u^L$ is the solution to~\eqref{eq:bemfgl}, $u$ is the solution to~\eqref{eq:bemfg}. Then, for each $\delta>0$ there exist versions of $u^L, u$ such that for all $\omega\in\Omega$ for each $x\in D_L$, $L=i\delta$, $i\in\zz_+$, $0<t\le T$, for any $0<\lambda<\frac{1}{4}$ holds
\begin{equation}\label{eq:udiffes}
|u^L(t,x)-u(t,x)|\le  C_{\delta}(\omega) C_\lambda  e^{-\frac{\lambda (L-|x|)^2}{t}} .
\end{equation}
\end{thm}

\begin{proof}
Consider
\begin{equation*}
\begin{split}
u^L(t,x)-u(t,x)=\Bigl(\int_{D_L} G^L(t,x,y)u_0(y)\,dy-\int_{\rr} p(t,x-y)u_0(y)\,dy\Bigr)\\
+\Bigl(\int_0^t \int_{D_L}G^L(t-s,x,y) f(s,y,u^L(s,y))\,dy\,ds\\
-\int_0^t \int_{\rr}p(t-s,x-y) f(s,y,u(s,y))\,dy\,ds\Bigr)\\
+\Bigl(\int_{D_L} \int_0^t G^L(t-s,x,y) \sigma(s,y)\,ds\,d\mu(y)\\
-\int_{\rr} \int_0^t p(t-s,x-y) \sigma(s,y)\,ds\,d\mu(y)\Bigr)\\
:=J_1+J_2+J_3.
\end{split}
\end{equation*}
by Lemma~\ref{lm:zetaes}, for some versions of the integrals, holds
\begin{equation*}
|J_3|\le  C_{\delta}(\omega) C_\lambda  e^{-\frac{\lambda (L-|x|)^2}{t}} .
\end{equation*}

For $J_1$, we have
\begin{eqnarray*}
|J_1|\le \int_{D_L} |G^L(t,x,y)-p(t,x-y)| |u_0(y)|\,dy+\int_{D_L^c} p(t,x-y)|u_0(y)|\,dy\\
\stackrel{\rm A\ref{assxu}}{\le}C_{u_0}(\omega) \int_{D_L} |G^L(t,x,y)-p(t,x-y)| \,dy+C_{u_0}(\omega)\int_{D_L^c} p(t,x-y)\,dy\\
:=J_{11}+J_{12}.
\end{eqnarray*}

We recall the Chernoff inequality (see, for instance, Example~2.1~\cite{wain19}). For random variable $X\sim N(\mu,\sigma^2)$ and $r\ge 0$ holds
\begin{equation}\label{eq:chern}
\pp\{X\ge \mu+r\}\le \exp\Bigl\{-\dfrac{r^2}{2\sigma^2}\Bigr\}.
\end{equation}
We have that $p(t,y-x)$, as a function of $y$, is the density function of the distribution $N(x,2t)$.

Thus, for $x\in D_L$, we obtain
\begin{eqnarray*}
\int_{D_L} p(t,x-y+4nL)\,dy\stackrel{v=y-4nL}{=} \int_{[-(4n+1)L,-(4n-1)L]} p(t,x-v)\,dv\\
\stackrel{\eqref{eq:chern}}{\le} \exp\Bigl\{-\dfrac{((4|n|-1)L-|x|)^2}{4t}\Bigr\}.
\end{eqnarray*}
In a similar way, we consider $p(t,x+y+(4n+2)L)$, $p(t,x+y+2L)$ in~\eqref{eq:estgp}, and obtain that
\begin{equation}\label{eq:estj11}
J_{11}\le C(\omega)\exp\Bigl\{-\dfrac{(L-|x|)^2}{4t}\Bigr\}.
\end{equation}

Also, \eqref{eq:chern} implies that
\begin{equation}\label{eq:estj12}
J_{12} \le 2 C_{u_0}(\omega) \exp\Bigl\{-\dfrac{(L-|x|)^2}{4t}\Bigr\}.
\end{equation}

For $J_2$, we get
\begin{eqnarray*}
|J_2|\le \int_0^t \int_{D_L}|G^L(t-s,x,y)-p(t-s,x-y)| |f(s,y,u^L(s,y))|\,dy\,ds\\
+\int_0^t \int_{D_L}p(t-s,x-y) |f(s,y,u^L(s,y))-f(s,y,u(s,y))|\,dy\,ds\\
+\int_0^t \int_{D_L^c} p(t-s,x-y) |f(s,y,u^L(s,y))|\,dy\,ds\\
:=J_{21}+J_{22}+J_{23}.
\end{eqnarray*}

Similarly to~\eqref{eq:estj11}, we obtain
\begin{eqnarray*}
J_{21}\stackrel{\rm A\ref{assxfgb}}{\le} \int_0^t C_f\int_{D_L} |G^L(t-s,x,y)-p(t-s,x-y)| \,dy\,ds\\
{\le} C(\omega)\int_0^t \exp\Bigl\{-\dfrac{(L-|x|)^2}{4(t-s)}\Bigr\}\,ds\le C(\omega)\exp\Bigl\{-\dfrac{(L-|x|)^2}{4t}\Bigr\}.
\end{eqnarray*}

As in~\eqref{eq:estj12}, we get
\begin{eqnarray*}
J_{23}\stackrel{\rm A\ref{assxfgb}}{\le} \int_0^t C_f\int_{D_L^c}p(t-s,x-y) \,dy\,ds\le C(\omega)\exp\Bigl\{-\dfrac{(L-|x|)^2}{4t}\Bigr\}.
\end{eqnarray*}
Also
\begin{eqnarray*}
J_{22}\stackrel{\rm A\ref{assxfgyv}}{\le} L_{f}\int_0^t \int_{D_L}p(t-s,x-y) |u^L(s,y)-u(s,y)|\,dy\,ds.
\end{eqnarray*}

Gathering together all estimates of $J_1$, $J_2$, and $J_3$, we obtain that for any $\lambda<\frac{1}{4}$
\begin{equation*}
\begin{split}
|u^L(t,x)-u(t,x)|\ii_{\{x\in D_L\} }\le C_{\delta}(\omega)C_{\lambda}\exp\Bigl\{-\dfrac{\lambda(L-|x|)^2}{t}\Bigr\}\\
+C(\omega)\int_0^t \int_{\rr}p(t-s,x-y) |u^L(s,y)-u(s,y)|\ii_{\{ y\in D_L\} }\,dy\,ds.
\end{split}
\end{equation*}

For $|u^L(t,x))-u(t,x)|\ii_{|x|\le L}$, we will use the space-time Gronwall-type inequality (see Lemma~C.1.7~(d)~\cite{dalbook}), and obtain
\begin{equation}\label{eq:ugron}
|u^L(t,x)-u(t,x)|\le \xi_{\delta,\lambda,L}(t,x)+C(\omega)\int_0^t \int_{\rr} \mathcal{K}(t-s,x-y) \xi_{\delta,\lambda,L}(s,y)\,dy\,ds
\end{equation}
where $\mathcal{K}$ denotes the series with iterated space-time convolutions
\[
\mathcal{K}(t,x)=\sum_{k=1}^\infty p^{*k}(t,x)
\]
(see more details in Section~C.1.3~\cite{dalbook}) and
\[
\xi_{\delta,\lambda,L}(t,x)=C_{\delta}(\omega)C_{\lambda}\exp\Bigl\{-\dfrac{\lambda(L-|x|)^2}{t}\Bigr\}\ii_{\{ x\in D_L\} }.
\]
It is well-known that
\[
\int_{\rr} p(t_1,x-y) p(t_2,y)\,dy=p(t_1+t_2,x)\,ds,
\]
therefore,
\[
p^{*2}(t,x)=\int_0^t \int_{\rr} p(t-s,x-y) p(s,y)\,dy\,ds=\int_0^t p(t,x)\,ds= t p(t,x),
\]
and applying induction, we easily obtain that
\begin{eqnarray*}
p^{*k}(t,x)=\frac{t^{k-1}}{(k-1)!} p(t,x)\quad\Rightarrow\quad \mathcal{K}(t,x)=e^t p(t,x).
\end{eqnarray*}

Further, we get
\begin{eqnarray*}
\int_0^t \int_{\rr} \mathcal{K}(t-s,x-y) \xi_{\delta,\lambda,L}(s,y)\,dy\,ds\\
=C(\omega)\int_0^t \int_{\rr} e^{t-s} p(t-s,x-y) \exp\Bigl\{-\dfrac{\lambda(L-|y|)^2}{s}\Bigr\}\ii_{\{ y\in D_L\} }\,dy\,ds\\
=C(\omega)\int_0^t \int_{\rr} e^{t-s} \dfrac{1}{2\sqrt{\pi(t-s)}}\exp\Bigl\{-\dfrac{(x-y)^2}{4(t-s)}\Bigr\} \exp\Bigl\{-\dfrac{\lambda(L-|y|)^2}{s}\Bigr\}\ii_{\{ y\in D_L\} }\,dy\,ds\\
\le C(\omega)\dfrac{2}{\sqrt{\lambda}}\int_0^t \dfrac{\sqrt{\pi s}}{\sqrt{\lambda}}\int_{\rr} \dfrac{\sqrt{\lambda}}{\sqrt{\pi(t-s)}}\exp\Bigl\{-\dfrac{\lambda(x-y)^2}{t-s}\Bigr\} \dfrac{\sqrt{\lambda}}{\sqrt{\pi s}}\Bigl(\exp\Bigl\{-\dfrac{\lambda(y-L)^2}{s}\Bigr\}\\
+\exp\Bigl\{-\dfrac{\lambda(y+L)^2}{s}\Bigr\}\Bigr)\,dy\,ds.
\end{eqnarray*}
From the properties of the convolution of density functions of normal distributions, we get
\begin{eqnarray*}
\int_{\rr} \dfrac{\sqrt{\lambda}}{\sqrt{\pi(t-s)}}\exp\Bigl\{-\dfrac{\lambda(x-y)^2}{t-s}\Bigr\} \dfrac{\sqrt{\lambda}}{\sqrt{\pi s}}\exp\Bigl\{-\dfrac{\lambda(y-L)^2}{s}\Bigr\}\,dy\\
=\dfrac{\sqrt{\lambda}}{\sqrt{\pi t}}\exp\Bigl\{-\dfrac{\lambda(x-L)^2}{t}\Bigr\},\\
\int_{\rr} \dfrac{\sqrt{\lambda}}{\sqrt{\pi(t-s)}}\exp\Bigl\{-\dfrac{\lambda(x-y)^2}{t-s}\Bigr\} \dfrac{\sqrt{\lambda}}{\sqrt{\pi s}}\exp\Bigl\{-\dfrac{\lambda(y+L)^2}{s}\Bigr\}\,dy\\
=\dfrac{\sqrt{\lambda}}{\sqrt{\pi t}}\exp\Bigl\{-\dfrac{\lambda(x+L)^2}{t}\Bigr\},
\end{eqnarray*}
therefore

\begin{eqnarray*}
\int_0^t \int_{\rr} \mathcal{K}(t-s,x-y) \xi_{\delta,\lambda,L}(s,y)\,dy\,ds
\le C_{\lambda,\delta} (\omega) \exp\Bigl\{-\dfrac{\lambda(L-|x|)^2}{t}\Bigr\},
\end{eqnarray*}
and \eqref{eq:ugron} implies the result of our Theorem.
\end{proof}

\begin{remark}
Let SM $\mu$ be such that $\mu((a,b])\le C(\omega)$, $a,\ b\in\rr$. Then we can obtain \eqref{eq:udiffes} for arbitrary $L\in\rr_+$, i.e.
\[
|u^L(t,x)-u(t,x)|\le  C(\omega) C_\lambda  e^{-\frac{\lambda (L-|x|)^2}{t}} ,
\]
where $C(\omega)$ is independent of $L$. This follows from Remark~\ref{rem:nodelta}, because we used form $L=i\delta$ only in Lemma~\ref{lm:zetaes}.
\end{remark}

\begin{remark} For the heat equation with a multiplicative noise driven by the Wiener process, under some assumptions, Theorem~3.1~\cite{candil26} states that for all $(t,x)\in [0, T]\times D_L$, $p\ge 1$ holds
\begin{eqnarray*}
\|u^L(t,x)-u(t,x)\|_{{\sf L_p}(\Omega)}\\
\le C(1+\sup_{x\in\rr}|u_0(x)|)\Bigl(\exp\Bigl\{-\frac{(L-x)^2}{8t}\Bigr\}+\exp\Bigl\{-\frac{(L+x)^2}{8t}\Bigr\}\Bigr),
\end{eqnarray*}
where $C$ is independent of $L,x,t,u_0$. In our paper, we have a similar convergence rate for $u^L\to u$.

Also, note that Theorem~3.1~\cite{candil26} includes the cases of the mixed boundary conditions and the Neumann boundary conditions for the equation.
\end{remark}

\bibliographystyle{plain}
\bibliography{RadchenkoAsympHeatLInft_24July2026}

\end{document}